\documentclass[12pt]{amsart}
\usepackage{amssymb}
\usepackage{mathrsfs}
\usepackage{xspace}
\usepackage{enumerate}
\usepackage{cancel}
\usepackage{ifthen}
\usepackage[plain]{fancyref}
\usepackage{hyperref}
\usepackage[normalem]{ulem}
\usepackage{mathrsfs}
\usepackage{xcolor}
\usepackage[all]{xy} 
\CompileMatrices

\newcommand{\Aut}{\operatorname{Aut}}

\newcommand{\Hom}{\operatorname{Hom}}

\newcommand{\EE}{{\mathbb E}}

\newtheorem{thm}{Theorem}[section]
\newtheorem{lem}[thm]{Lemma}
\newtheorem{cor}[thm]{Corollary}
\newtheorem{prop}[thm]{Proposition}
\newtheorem{claim}{Claim}

\newtheorem*{claim*}{Claim}

\theoremstyle{definition}
\newtheorem{defn}[thm]{Definition}

\newtheorem{question}[thm]{Question}
\newtheorem{conj}[thm]{Conjecture}

\theoremstyle{remark}
\newtheorem{rem}[thm]{Remark}

\newcommand*\fancyrefthmlabelprefix{thm}\frefformat{plain}{\fancyrefthmlabelprefix}{Theorem~#1}
\newcommand*\fancyreflemlabelprefix{lem}\frefformat{plain}{\fancyreflemlabelprefix}{Lemma~#1}
\newcommand*\fancyrefproplabelprefix{prop}\frefformat{plain}{\fancyrefproplabelprefix}{Proposition~#1}
\newcommand*\fancyrefcorlabelprefix{cor}\frefformat{plain}{\fancyrefcorlabelprefix}{Corollary~#1}
\newcommand*\fancyrefclaimlabelprefix{claim}\frefformat{plain}{\fancyrefclaimlabelprefix}{Claim~#1}
\newcommand*\fancyreffactlabelprefix{fact}\frefformat{plain}{\fancyreffactlabelprefix}{Fact~#1}
\newcommand*\fancyrefquestionlabelprefix{question}\frefformat{plain}{\fancyrefquestionlabelprefix}{Question~#1}
\newcommand*\fancyrefconjlabelprefix{conj}\frefformat{plain}{\fancyrefconjlabelprefix}{Conjecture~#1}
\newcommand*\fancyrefdefnlabelprefix{defn}\frefformat{plain}{\fancyrefdefnlabelprefix}{Definition~#1}
\newcommand*\fancyrefconstlabelprefix{const}\frefformat{plain}{\fancyrefconstlabelprefix}{Construction~#1}
\newcommand*\fancyrefsetuplabelprefix{setup}\frefformat{plain}{\fancyrefsetuplabelprefix}{Setup~#1}
\newcommand*\fancyrefexlabelprefix{ex}\frefformat{plain}{\fancyrefexlabelprefix}{Example~#1}
\newcommand*\fancyrefremlabelprefix{rem}\frefformat{plain}{\fancyrefremlabelprefix}{Remark~#1}
\frefformat{plain}{\fancyrefeqlabelprefix}{(#1)}
\newcommand*\fancyrefitemlabelprefix{item}\frefformat{plain}{\fancyrefitemlabelprefix}{(#1)}

\def\repeat#1#2 {\expandafter\gdef\csname B#1\endcsname {\mathbb{#1}}
  \ifthenelse{\equal{#2}{*}}{}{\repeat #2 }}
\repeat ABCDEFGHIJKLMNOPQRSTUVWXYZ*
\def\repeat#1#2 {\expandafter\gdef\csname rm#1\endcsname {\mathrm{#1}}
  \ifthenelse{\equal{#2}{*}}{}{\repeat #2 }}
\repeat ABCDEFGHIJKLMNOPQRSTUVWXYZ*
\def\repeat#1#2 {\expandafter\gdef\csname C#1\endcsname {\mathcal{#1}}
  \ifthenelse{\equal{#2}{*}}{}{\repeat #2 }}
\repeat ABCDEFGHIJKLMNOPQRSTUVWXYZ*
\def\repeat#1#2 {\expandafter\gdef\csname bf#1\endcsname {\boldsymbol{#1}}
  \ifthenelse{\equal{#2}{*}}{}{\repeat #2 }}
\repeat abcdefghijklmnopqrstuvwxyzABCDEFGHIJKLMNOPQRSTUVWXYZ*
\def\repeat#1#2 {\expandafter\gdef\csname scr#1\endcsname {\mathscr{#1}}
  \ifthenelse{\equal{#2}{*}}{}{\repeat #2 }}
\repeat abcdABCDEFGHIJKLMNOPQRSTUVWXYZ*
\DeclareMathOperator{\gr}{gr}
\DeclareMathOperator{\Homeo}{Homeo}
\DeclareMathOperator{\rank}{rk}

\DeclareMathOperator{\Stab}{Stab}
\DeclareMathOperator{\Tor}{Tor}

\newcommand{\HomSheaf}{\ensuremath{\underline{\CH\kern-1pt om}}}
\newcommand\gen[1]{{\ensuremath\langle #1\rangle}} 
\newcommand{\dcm}[1][d]{$#1$-$\mathfrak{cm}_p$\xspace} 
\newcommand\Orientation[1][]{\ifthenelse{\equal{#1}{}}{\CO_{p}}{\CO_{p,#1}}}
\newcommand{\cm}{$\mathfrak{cm}_p$} 
\newcommand{\Gcm}{$G\!\curvearrowright\!\mathfrak{cm}_p$} 
\newcommand{\Gdcm}{$G\!\curvearrowright\! d$-$\mathfrak{cm}_p$} 

\newcommand\Nth[1][N] {{{}^{(#1)}\kern-2pt}}
\newcommand{\Fix}[2]{{#1}^{#2}}
\newcommand{\isom}{\cong}

\begin{document}
\title[Number of stabilizers]
{On the number of stabilizer subgroups in a finite group acting on a manifold}
\author[B. Csik\'os \and I. Mundet i Riera \and
L. Pyber \and E. Szab\'o]{Bal\'azs Csik\'os \and Ignasi Mundet i Riera \and
	L\'aszl\'o Pyber \and Endre Szab\'o}
\maketitle
\begin{abstract}
  If a finite $p$-group $G$ acts continuously on a compact topological
  manifold $M$ then, with some bound $C$ depending on $M$ alone,
  $G$ has a subgroup $H$ of index at most $C$ such that the $H$-action
  on $M$ has at most $C$ stabilizer subgroups.
  This result plays a crucial role in the proof of a deep conjecture of Ghys
  \cite{Ghys-conjecture}.
\end{abstract}
\section{Introduction}
\label{sec:introduction}

It is a basic question how the topology of a manifold restricts the
algebraic and geometric complexity of finite group actions on the
manifold.
The following deep conjecture of Ghys points in this direction.
\begin{conj}
  For each compact differentiable manifold $M$ there is a constant $I$
  such that every finite subgroup $G$ of the diffeomorphism group of $M$
  has a nilpotent subgroup of index at most $I$.
\end{conj}
Originally, about twenty years ago,
Ghys conjectured something stronger:
the group $G$ should contain an abelian subgroup of bounded index.
That was disproved in \cite{Ghys-conjecture-false},
so Ghys modified the conjecture to the one above.
The conjecture (in a more general form, for the homeomorphism group of
topological manifolds)
has been proved recently
in \cite{Ghys-conjecture}. One of the main technical tools used in
\cite{Ghys-conjecture}
is a weaker version of
\fref{thm:for-p-groups,topological-manifolds} below.
\fref{thm:for-p-groups,topological-manifolds} is likely to have a
number of further applications,
see \cite{Conjugacy-classes-of-stabilizers} and \cite{Jordan-homeo}
discussed below.

Classical Smith theory gives strong restrictions on the topology of
fixed point submanifolds,
and Borel's Fixed Point Formula describes locally the possible
ways how these submanifolds can meet at any single point.
To complete this picture,
our paper focuses on bounds on the number of stabilizer subgroups.
As Smith theory and Borel's formula, our results apply
only to actions of finite $p$-groups.

\begin{defn}
  Let $G$ be a group acting on a set $X$.
  Denote by $G_x$ the  stabilizer subgroup
  of an element $x\in X$.
  For subgroups $H\le G$, we denote by $X^H$ the
  fixed point subset of $H$.
  We will study the set of all stabilizer subgroups:
  $$
  \Stab(G,X) = \left\{G_x\,\big|\,x\in X\right\}.
  $$
\end{defn}

Our main theorem is the following.

\begin{thm} \label{thm:for-p-groups,topological-manifolds}
  Let $M$ be a topological manifold such that
  $H_*(M;\BZ)$ is finitely generated (as an abelian group).
  Then there is a number $C$
  depending on $\dim(M)$ and $H_*(M;\BZ)$ with the following property.
  For every prime $p$, every finite $p$-group $G$ acting continuously
  on $M$
  has a characteristic subgroup $H\le G$ of index at most $C$
  containing the center of $G$
  such that
  $$
  \big|\Stab(H,M)\big|\le C.
  $$
\end{thm}

\begin{rem}
  \label{rem:dihedral}
  In this theorem, it is necessary to consider a subgroup $H$. Indeed, let $p$ be a fixed prime, $n>0$ an arbitrary natural number, \[T^{p-1}=\{(z_1,\dots,z_p)\in \mathbb C^p:z_1\cdots z_p=1, |z_i|=1 \text{ for all }i\}\] a $(p-1)$-dimensional torus. The subgroup $H=\{(z_1,\dots,z_p)\in T^{p-1}: (z_i)^{p^n}=1  \text{ for all }i\}$ of  $T^{p-1}$ acts on $T^{p-1}$ by left translations, the cyclic group $A$ of order $p$ acts both on $T^{p-1}$ and $H$ by cyclic permutations of the coordinates $z_i$. These actions generate an action of the semidirect product  $G=A\ltimes H$ on  $T^{p-1}$. This action has $p^{(p-1)n}$ different stabilizer subgroups, the conjugates of $A$ in $G$, however, each of these intersects the subgroup $H\triangleleft G$ in the trivial group.
\end{rem}

Classical Smith theory allows one to bound the cohomology of fixed point
submanifolds, and their complements (see \fref{prop:Smith-Floyd}).
Combining this with \fref{thm:for-p-groups,topological-manifolds},
we can bound the cohomology of many other interesting submanifolds.
For example, the following corollary, which is used in  \cite{Ghys-conjecture}, provides a technique to reduce a statement on arbitrary actions to the study of free actions.

\begin{cor} \label{cor:free-part-has-bounded-homology}
  Let $M$ be a topological manifold such that
  $H_*(M;\BZ)$ is finitely generated (as an abelian group).
  Then there is a number $\mathring C$
  depending on $\dim(M)$ and $H_*(M;\BZ)$ with the following property.
  For every prime $p$, every finite $p$-group $G$ acting continuously
  on $M$
  has a characteristic subgroup $H\le G$ of index at most
  $\mathring C$
  containing the center of $G$
  such that if $K$ is minimal among the subgroups in $\Stab(H,M)$,
  then the subset $\mathring M^K=\{x\in M\;|\; H_x=K\}$ is open and
  \[
    \dim H_*(\mathring M^K;\BF_p)\le \mathring C.
  \]
\end{cor}

\begin{rem}
  If $M$ is connected and the subgroup $H\le G$ given
  by the corollary acts effectively on $M$, then $\{1\}$ is the unique minimal element in
  $\Stab(H,M)$, and $\mathring M^{\{1\}}$ is the largest open subset
  in $M$ where the $H$-action is free.
\end{rem}

\begin{rem} \label{rem:free-part-has-bounded-homology}
  Similar bound holds for all subgroups $K\in\Stab(H,M)$. In general,
  $\mathring M^K$ may not be a manifold, just a disjoint union of cohomology
  manifolds, and we cannot control its singular homology.
  Instead,  we can bound its Borel-Moore homology with compact
  support:
  \[
    \dim H^c_*(\mathring M^K;\BF_p)\le \mathring C.
  \]
\end{rem}

The above results have many important applications.
\fref{thm:for-p-groups,topological-manifolds}
and a variation of
\fref{cor:free-part-has-bounded-homology} play a crucial role in the
proof of the Ghys conjecture \cite{Ghys-conjecture}.
Further applications of \fref{thm:for-p-groups,topological-manifolds}
are presented in \cite{Conjugacy-classes-of-stabilizers}. Finally,
\fref{thm:for-p-groups,topological-manifolds}   is a key ingredient
in the proof of the almost fixed point property for compact topological manifolds
with non-zero Euler characteristic in \cite{Jordan-homeo}.
The latter means that if 
$M$ is a compact manifold with nonzero Euler characteristic then
there exists a constant $C$ with this property: for any continuous action
of a finite group $G$ on $M$ there exists a point $x\in M$ such
that the index of $G_x$ in $G$ is at most $C$. Note that, besides
\fref{thm:for-p-groups,topological-manifolds}, the proof of the previous
result uses the validity of Ghys's original conjecture for
homeomorphism groups of compact manifolds with nonzero Euler characteristic.

Recall that the \emph{rank} of a finite group $G$ is the minimal
integer $r$ such that every subgroup $H$ of $G$ is $r$-generated. A
result of Mann and Su \cite{Mann_Su} gives an upper bound on the rank
of an elementary abelian $p$-group $G$ acting effectively on a compact
manifold $M$ in terms of the dimension and
the $\BF_p$-cohomology of $M$.
In order to prove \fref{thm:for-p-groups,topological-manifolds}
we generalize this result to not necessarily compact manifolds
and to arbitrary finite groups $G$.
\begin{thm}\label{thm:Mann-Su_v0}
	Let $G$ be  a finite group  acting continuously
	and effectively on a topological manifold $M$ such that $ H_*(M;\BZ)$ is finitely generated.
	Then the rank of $G$ is bounded in terms of $\dim(M)$ and $
	H_*(M;\BZ)$.
\end{thm}

\fref{thm:for-p-groups,topological-manifolds} is about actions of
finite $p$-groups.
One is tempted to ask the following far reaching question:

\begin{question}
\label{question:question}
  Let $M$ be a compact topological manifold.
  Can one find a bound $\tilde C(M)$ such that
  each finite group $G$ acting continuously on $M$
  has a subgroup $H\le G$ of index at most $\tilde C(M)$
  such that
  $$
  \big|\Stab(H,M)\big|\le \tilde C(M)?
  $$
\end{question}

This question is probably substantially more difficult than
\fref{thm:for-p-groups,topological-manifolds}, and might be out of reach
with the tools used in this paper.
Nevertheless, this question is answered affirmatively in
\cite{mundet:topological-rigidity-of-tori} for the $n$-dimensional torus or,
more generally, for closed oriented $n$-manifolds admitting a
continuous map of nonzero degree onto an $n$-dimensional torus.


The structure of the paper is the following. In
\fref{sec:preliminaries} we summarize some facts from Smith theory, on
equivariant cohomology, and on group cohomology that will be used
later. Section \ref{sec:Mann_Su} is devoted to the proof of
\fref{thm:Mann-Su_v0}. The proof of
\fref{thm:for-p-groups,topological-manifolds} is divided into two
parts. First we give a topological proof for elementary abelian
$p$-groups in \fref{sec:proof-p-groups}, then the case of arbitrary
$p$-groups is settled  inductively in \fref{sec:proof-general}. The
last section contains the proof of
\fref{cor:free-part-has-bounded-homology} and \fref{rem:free-part-has-bounded-homology}.

\subsection*{Acknowledgement}
L.~Pyber and E. Szab\'o were supported by the National Research, Development and
Innovation Office (NKFIH) Grant K138596. E.~Szab\'o was also supported
by the NKFIH Grant K120697. B.~Csik\'os was supported by the NKFIH Grant K128862.
The project leading to this application has received funding from the European Research Council (ERC) under the European Union's Horizon 2020 research and innovation programme (grant agreement No 741420).
The research of I. Mundet i Riera was partially supported by the grant 
PID2019-104047GB-I00 from the Spanish Ministeri de Ci\`encia i Innovaci\'o.

\section{Preliminaries}
\label{sec:preliminaries}
Throughout this paper, $p$ always denotes a prime number, $\BF$ will denote an arbitrary field, $\BF_p$ the field of $p$ elements.
A $d$-dimensional  cohomology manifold $M$ over $\BF_p$ will be called shortly  a {\cm} or {\dcm}. The dimension of a {\cm} $M$ will be denoted by $\dim_p(M)$. If, in addition, $M$ is equipped with an (effective) continuous action of the group $G$, then $M$ is called briefly an (effective) {\Gcm} or {\Gdcm}. Cohomology will always mean sheaf cohomology.

\subsection{Orientation sheaf}
 Every {\cm}
$M$ comes with an
\emph{orientation sheaf} $\Orientation$
(see \cite[Definition~V-9.1]{bredon2012sheaf} for a definition),
or with more precise notation $\Orientation[M]$,
if we need to emphasize $M$. If $\tilde M=\bigcup_{i=1}^kM_i$ is the disjoint union of the {\cm}'s $M_i$, not necessarily of the same dimension, then the orientation sheaf $\Orientation[\tilde M]$ of $\tilde M$ is defined as the sheaf whose restriction onto $M_i$ is $\Orientation[M_i]$ for all $i$.


Many of our constructions involving the orientation sheaf will be \emph{functorial} with respect to open embeddings due to the following simple proposition.

\begin{prop} \label{prop:Orientation-canonical}
	For any open embedding $f\colon N\to M$ of a {\cm} $N$ into a {\cm} $M$,
	there is a natural isomorphism
	$$
	\Orientation[N]\stackrel\isom\longrightarrow
	f^*\Orientation[M].
	$$
\end{prop}

\begin{defn}
	If $G$ is a discrete group, $X$ is a topological space equipped with a continuous action of $G$, or shortly a \emph{$G$-space}, and $R$ is a commutative unital ring, then a sheaf of $R$-modules $\pi\colon \mathcal A\to X$  on $X$ is a \emph{$G$-sheaf} if we are given a continuous action $\tau$ of $G$ on the sheaf space $\mathcal A$ for which $\tau_g(\mathcal A_x)={\mathcal A_{gx}}$ and the restriction of $\tau_g|_{\mathcal A_x}\colon {\mathcal A_x} \to {\mathcal A_{gx}} $ is a module isomorphism for every $g\in G$ and $x\in X$.
\end{defn}

The most important examples of $G$-sheaves in this paper are the
orientation sheaves of  cohomology manifolds with a group action.
\begin{prop}\label{prop:G-structure on orientaion_sheaf}
  For a {\Gcm} $M$, the action of $G$ on $M$ can be lifted in a unique natural way to a $G$-action on the sheaf space of $\Orientation[M]$, providing a natural $G$-sheaf structure on $\Orientation[M]$.
\end{prop}
\begin{proof}
   Follows immediately from the functoriality of the orientation sheaf.
\end{proof}

\begin{lem}\label{lem:H_0_H^d_duality}
	If a $d$-{\cm} $M$ has $C$ connected components, then $C=\dim H_c^d(M, \Orientation)$.
\end{lem}
\begin{proof}
By Poincar\'e duality \cite[V-9.2]{bredon2012sheaf}, $H_c^d(M, \Orientation)$ is isomorphic to the Borel-Moore homology $H_0^c(M, \BF_p)$ with compact support, and the latter is isomorphic to the free $\BF_p$-module generated by the connected components of $M$, see \cite[V-5.14]{bredon2012sheaf}. This implies the equality.
\end{proof}

\subsection{Smith theory}
Smith theory studies the fixed point set $F$ of a $\BZ_p$-action on a {\cm} $M$. The connected components of $F$ are known to be {\cm}'s, and a theorem of Smith and Floyd provides a bound on $\dim H_c^*\big(F,\Orientation\big)$ and $\dim H_c^*\big(M\setminus F,\Orientation\big)$ in terms of $\dim H_c^*\big(M,\Orientation\big)$. These results extend to arbitrary $p$-group actions by a simple induction.
\begin{prop} \label{prop:Smith-Floyd}
	Let $G$ be a finite $p$-group,
	and $M$ a {\Gcm}.
	Then each connected component $F$ of $\Fix{M}{G}$ is a {\cm},
	$\dim_p(F)\le\dim_p(M)$,
	and we have an isomorphism
	\begin{equation}\label{eq:O_F}
	\Orientation[M]\big|_F\isom\Orientation[F]
	\end{equation}
	which is natural with respect to open embeddings.
        If $M$ is connected and the $G$-action is non-trivial
        then $\dim_p(F)\le\dim_p(M)-1$.
	
	Furthermore,  we have the inequalities
\begin{align}\label{eq:Smith_Floyd_1}
	\dim H_c^*\big(\Fix{M}{G};\Orientation\big) &\le
	\dim H_c^*(M;\Orientation),
\\
	\dim H_c^*\big(M\setminus\Fix{M}{G};\Orientation\big) &\le
	2\,\dim H_c^*(M;\Orientation).\label{eq:Smith_Floyd_2}
\end{align}
\end{prop}
\begin{proof}
	First we prove the statements on $F$ and inequality \fref{eq:Smith_Floyd_1} by induction on $|G|$.

	Assume first that $G\isom\BZ_p$. Theorem V-20.1 in
	\cite{bredon2012sheaf}
	implies that $F$ is a homology manifold over $\BF_p$ of dimension
	$\dim_p(F)\le\dim_p(M)$,
	and gives us an isomorphism \eqref{eq:O_F} which depends only on the choice of a generator of $G$. Once such a generator is fixed, the isomorphisms corresponding to that generator behave in a natural way with respect to open embeddings.
	It is stated after the proof of
	\cite[Theorem V-20.2]{bredon2012sheaf},
	that $F$ is a {\cm} as well.
	
	Next we prove that if $M$ is connected and the $G$-action is nontrivial then
        $\dim_p(F)\ne\dim_p(M)$.  Suppose, on the contrary, that $\dim_p(F)=\dim_p(M)$. Then
	$F$ is open by \cite[Corollary V-16.19]{bredon2012sheaf},
	and it is closed as well. As $M$ is connected and $F\neq \emptyset$, we must have $F=M$.
	This contradicts the effectiveness of the action.
	
	Finally, inequality \fref{eq:Smith_Floyd_1} for $G=\BZ_p$ is guaranteed by the Smith--Floyd theorem \cite[Theorem II-19.7]{bredon2012sheaf}.
	
	Next we do the induction step.
	Choose a subgroup $A\isom\BZ_p$ in the center of $G$. If $F$ is a connected component of $M^G$, and $N$ is the connected component of $M^A$ containing $F$, then $N$ and $GN\subseteq M^A$ are {\cm}'s and we have $\dim_p(N)=\dim_p(GN)<\dim_p(M)$ by the induction hypothesis applied to the $A$-action on $M$. As the action of $A$ on $GN$ is trivial, there is a canonical $G/A$ action on $GN$ and $F$ is a connected component of $(GN)^{G/A}$. Applying the induction hypothesis again to this $(G/A)$-action, we see that our statements hold for $F$. Observe that the recursive construction of the isomorphism \eqref{eq:O_F} depends on the choice of a composition series of $G$, and a choice of a generator in each composition factor, but if these data are fixed for $G$, then the resulting isomorphisms will behave naturally with respect to open embeddings.
	
	As for inequality  \fref{eq:Smith_Floyd_1}, the induction hypothesis yields
	\begin{align*}
	\dim H_c^*\big(\Fix{M}{G};\Orientation\big) &=	\dim H_c^*\big((\Fix{M}{A})^{G/A};\Orientation\big)\le \dim H_c^*\big(\Fix{M}{A};\Orientation\big)\\&\le
\dim H_c^*(M;\Orientation).	
	\end{align*}

	Inequality \fref{eq:Smith_Floyd_2} follows from inequality \fref{eq:Smith_Floyd_1}
by the long exact sequence of the pair $(M,M\setminus M^G)$ for cohomology with compact support
(see \cite[II-10.3]{bredon2012sheaf}).
\end{proof}

\subsection{Equivariant cohomology}

Here we briefly recall the definition of equivariant cohomology with
compact support, denoted by $H_{G,c}^*(M;-)$,
group cohomology, denoted by $H^*(G;-)$,
and some basic tools for their computation.
Our main goal is to give a proof of the following three propositions,
that will be needed later.

\begin{prop}\label{prop:long-exact-sequence-of-closed-subset-cor}
	Let $G$ be a finite $p$-group, $M$ a {\Gcm}.
	Assume that the fixed point sets of the normal subgroups
	$H_1,\dots,H_k$ of $G$ are pairwise disjoint.
	If  $F=\bigcup_{i=1}^k M^{H_i}$ is their union,
	and $U=M\setminus F$, then
	we have a long exact sequence 
	$$
	\cdots
	H_{G,c}^d\big(U;\Orientation[U]\big)\to
	H_{G,c}^d\big(M;\Orientation[M]\big)\to
	H_{G,c}^d\big(F;\Orientation[F]\big)\to
	H_{G,c}^{d+1}\big(U;\Orientation[U]\big)\cdots,
	$$
	which is functorial for $G$-equivariant open embeddings.
\end{prop}
\begin{prop}\label{prop:free-action-cohomology-cor}
	Let $G$ be a finite group, and $M$ a {\Gcm}.
	\begin{enumerate}[\indent(a)]
		\item \label{item:8}
		If the action of $G$ is free, then
		$$
		H_{G,c}^*\big(M;\Orientation[M]\big) \isom
		H_c^*\big(M/G;\Orientation[M/G]\big).
		$$
		\item \label{item:9}
		If $G=K\times L$ and the stabilizer of each point of $M$ is the subgroup $L$,
		then
		$$
		H_{G,c}^*\big(M;\Orientation[M]\big) \isom
		H^*(L;\BF_p)\otimes H_c^*\big(M/G;\Orientation[M/G]\big).
		$$
	\end{enumerate}
	These isomorphisms are natural transformations
	for $G$-equivariant open embeddings.
\end{prop}

\begin{prop}
	\label{prop:Leray-spectral-sequence}
	Let $G$ be a finite group, and $M$ a {\Gcm}.
	Then
	there exists a spectral sequence
	\begin{equation}
	\label{eq:3}
	\BE_t^{i,j} \Longrightarrow H_{G,c}^{i+j}(M;\Orientation),
	\quad\quad
	\BE_2^{i,j}=H^i\big(G;H_c^j(M;\Orientation)\big),
	\end{equation}
	which is functorial in $M$ with respect to $G$-equivariant open embeddings.
\end{prop}

 Actually, it will be more convenient to formulate and prove  these propositions in a more general form. Here are the details.
\smallskip

The following construction is due to A. Borel
(see \cite[Section~IV-3.1]{borel1960seminar}).
Let $G$ be a topological group and choose a universal principal
$G$-bundle $EG\to BG$. If $X$ is a $G$-space, then there is an
associated bundle $(EG\times X)/G\to BG$ with fiber $X$, the total
space $X_G=(EG\times X)/G$ of which is called the \emph{homotopy  quotient of $X$}.
Assuming $G$ is a compact Lie group,
Borel approximates $EG$ by a compact $N$-universal principal bundle.
Using that he obtains an approximation $X_G^{(N)}$ of $X_G$,
and studies its cohomology groups (or rather their limits as $N\to\infty$).
These are called the \emph{equivariant cohomology groups of $X$}.
In this paper, we use the version which uses cohomology with compact
support, $H_c^*(\tilde X_G,-)$.
However, it is more convenient for us to use a different (but equivalent)
approach due to Grothendieck \cite{Grothendieck}.

Suppose now that $G$ is a discrete group.
If we are given a $G$-sheaf $\mathcal A$ of $R$-modules on the
$G$-space $X$,
then we can consider the $R$-module $\Gamma_{c}^G(\mathcal A)$ of
$G$-equivariant sections of $\mathcal A$ with compact support.
The functor $\Gamma_{c}^G(-)$ is left exact, so it gives rise to the
right derived functors $\rmR^i\Gamma_{c}^G(-)$.

\begin{defn}[Grothendieck {\cite[Section~5.7]{Grothendieck}}]
  The $i$-dimensional equivariant sheaf cohomology of the
    $G$-sheaf $\mathcal A$ of $R$-modules on the $G$-space $X$ with
    compact support is the $R$-module\footnote{ Grothendieck used the
      notation $H_{c}^i(X;G,\mathcal A)$ for this group.}
  \[
    H_{G,c}^i(X;\mathcal A)=\rmR^i\Gamma_{c}^G(\mathcal A).
  \]
\end{defn}

We'll need another derived functor, also defined by Grothendieck.

\begin{defn}
  Let $X$ be a $G$-space,
  and $f\colon X\to X/G$ be the quotient map.
  For a $G$-sheaf $\CA$ on $X$ let $f_*^G(\CA)$ denote
  the subsheaf of $G$-invariant elements of the sheaf $f_*(\CA)$.
  The functor $f^G_*$ is left exact,  $\rmR^if^G_*$ denotes its $i$-th derived functor.
\end{defn}

The notion of group cohomology is an important special case of equivariant cohomology.
\begin{defn}\label{defn:Group-cohomology}
  Let $G$ be a discrete group, $V$ an $RG$-module.
  Let  $\mathcal V$ be the $G$-sheaf on a one point space $\{p\}$,
  with stalk $V$ at $p$, where $G$ acts on the stalk via the
  $RG$-module structure.
  Then the \emph{group cohomology of  $G$ with coefficients in  $V$} is
  $$
  H^*(G;V)=H_{G,c}^*\big(\{p\};\mathcal V\big).
  $$	
\end{defn}

The cohomology ring\footnote{
  In this paper, we neither define, nor use the ring structure,
  but it is easier to describe these cohomology groups in terms of rings.}
of abelian groups is well-known.
We collect here some useful facts.

\begin{prop} \label{prop:cohomology-of-elementary-abelian-p-groups}
	Let $G$ be an elementary abelian $p$-group of rank $r$.
	\begin{enumerate}[\indent(a)]
		\item \label{item:10}
		For $p=2$, we have $H^*\big(G;\BF_2\big)=\BF_2[t_1,\dots,t_r]$
		with $\deg(t_i)=1$ for all $i$.
		\item \label{item:11}
		If  $p\ge3$, then
		$H^1(G;\BF_p)$ is naturally isomorphic to $\Hom(G,\BF_p)$,
		the Bockstein homomorphism
		$\beta^1\colon H^1(G;\BF_p)\hookrightarrow H^2(G;\BF_p)$ is injective,
		and
		$$
		H^*(G;\BF_p) =
		\Lambda^*\big(H^1(G;\BF_p)\big)\otimes
		S^*\big(\beta^1(H^1(G;\BF_p))\big)
		$$
		where, for a vector space $V$ over $\BF_p$,
		$\Lambda^*(V)$ and $S^*(V)$ denote the Grassmann algebra and the
		symmetric algebra of $V$.
		In particular,
		for all $d$, we have a natural isomorphism
		$$
		H^d(G;\BF_p) =
		\bigoplus_{l+2s=d}
		\Lambda^l\big(\Hom(G;\BF_p)\big) \otimes
		S^s\big(\Hom(G;\BF_p)\big).
		$$
		\item \label{item:12} For all $p$, we have
		$$
		\dim H^d(G;\BF_p) = \binom{d+r-1}{d} = \binom{d+r-1}{r-1},
		$$
		and
		$$
		\dim H^d(G;\BF_p) = \frac{d+1}{r+d}\,\dim H^{d+1}(G;\BF_p).
		$$
		\item \label{item:14}
		For any finite dimensional $\BF_pG$-module $V$,
		we have
		$$
		\dim H^d(G;V) \le \dim H^d(G;\BF_p)\dim(V).
		$$
	\end{enumerate}
\end{prop}
\begin{proof}
	The cohomology ring of $\BZ_p$ is described in
	\cite[IV-2.1(3)]{borel1960seminar}. Statements (a), (b), (c) for $G=\BZ_p^d$
	follow from the K\"unneth formula.
	
	To prove (d) we recall that any short exact sequence of $\BF_pG$-modules $0\to W'\to W\to W''\to 0$ gives rise to a long exact sequence
	\begin{align*}
	0\to H^0(G;W')&\to H^0(G;W)\to H^0(G;W'')\to \\
	&\to H^1(G;W')\to H^1(G;W)\to H^1(G;W'')\to\dots,
	\end{align*}
    which implies the inequality
    \[\dim H^i(G;W)\leq \dim H^i(G;W')+\dim H^i(G;W'').\]

    It is well-known that a finite $p$-group has only unipotent
    representations in characteristic $p$,
    hence $V$ has a filtration $V=V_{d}>V_{d-1}>\dots>V_{0}=0$
    such that the induced $G$-action on $\gr V$ is trivial. Then the previous inequality yields
     \begin{align*}\dim H^i(G;V_i)&\leq \dim H^i(G;V_{i-1})+\dim H^i(G;V_i/V_{i-1})\\&=\dim H^i(G;V_{i-1})+\dim H^i(G;\BF_p),
     \end{align*}
     and a simple induction completes the proof.
\end{proof}

Now we construct an equivariant analogue of the fundamental long exact cohomology sequence associated to a closed subspace of a space.

\begin{lem}\label{lem:injective_stalk}
	Let $G$ be a finite group, $\CI$ be an injective $G$-sheaf of $R$-modules over the $G$-space $X$. Then the stalk $\CI_x$ of $\CI$ at $x$ is an injective $RG_x$-module for any $x\in X$.
\end{lem}
\begin{proof}
  By \cite[Proposition~5.1.2]{Grothendieck} we can embed $\CI$ into an
  injective $G$-sheaf of the form $\CJ=\prod_{x\in X}\CJ(x)$, where each
  $\CJ(x)$ is a skyscraper sheaf at $x$ with values in some injective $RG_x$-module.
  Since $\CI$ is injective, it must be a direct summand, hence the
  stalk $\CI_x$ is a direct summand of $\CJ_x$. Therefore $\CI_x$ is
  injective as an $RG_x$-module.
\end{proof}

\begin{defn}[{\cite[I-2.6]{bredon2012sheaf}}]
Let $A\subseteq X$ be a locally closed subset of $X$, $\CA$ a sheaf of $R$-modules on $A$. The \emph{extension $\CA^X$ of $\CA$ onto $X$ by zero} is the sheaf on $X$ determined uniquely by the conditions that its restriction onto $A$ is the sheaf $\CA$, while its restriction onto ${X\setminus A}$ is the $0$ sheaf on $X\setminus A$.
\end{defn}
 If $X$ is a $G$-space and $A$ is $G$-invariant, then  extension by 0 is an exact functor from the category of $G$-sheaves of $R$-modules on $A$ to the category of $G$-sheaves of $R$-modules on $X$.
\begin{lem}\label{lem:effaceable}
	Let $G$ be a finite group, $X$ be a locally compact  Hausdorff $G$-space, $A\subseteq X$ a  $G$-invariant closed or open subset, $\CI$ an injective $G$-sheaf on $A$. Then $H^n_{G,c}(X,\CI^X)=0$ for $n>0$.
\end{lem}

\begin{proof} Let $f\colon X\to Y=X/G$ be the quotient map.
  By \cite[Section 5.7]{Grothendieck}, there is a spectral sequence $I_t^{i,j}\Longrightarrow H^{i+j}_{G,c}(X,\CI^X)$ with second page
\[
I_2^{i,j}=H^i_c(Y,\rmR^jf_*^G(\CI^X)).
\]
By \cite[Theorem 5.3.1]{Grothendieck} $\rmR^jf_*^G(\CI^X)$ is a sheaf over $Y$, the stalk of which at $y=f(x)$ is $H^j(G_x,\CI^X_x)$. Since $\CI^X_x=\CI_x$ is injective for $x\in A$ by \fref{lem:injective_stalk}, and $\CI^X_x=0$ for $x\notin A$,  $\rmR^jf_*^G(\CI^X)$ is the $0$ sheaf for $j>0$. This implies that the spectral sequence degenerates at the second page and
\begin{equation*}
H^n_{G,c}(X,\CI^X)=H^n_c(Y,f_*^G(\CI^X))=H^n_c(Y,f_*^G(\CI)^Y).
\end{equation*}
As $A$ is $G$-invariant and either open or closed in $X$,
$B=f(A)$ is open or closed in $Y$. Thus, according to  \cite[Theorem II-10.1]{bredon2012sheaf}, we have
\[
H^n_c(Y,f_*^G(\CI)^Y)\isom H^n_c(B,f_*^G(\CI)).
\]
As $\CI$ is injective, $f_*^G(\CI)$ is flabby by \cite[Proposition 5.1.3 and its Corollary]{Grothendieck}, which implies $H^n_c(B,f_*^G(\CI))=0$ for $n>0$.
\end{proof}

\begin{lem}\label{lem:long-exact-sequence-isomorphism}
	Let $G$ be a finite group, $X$ be a locally compact  Hausdorff $G$-space, $A\subseteq X$ a $G$-invariant closed or open subset. Then for any $G$-sheaf of $R$-modules on $A$, we have a natural isomorphism
\[
H^*_{G,c}(X,\CA^X)\isom H^*_{G,c}(A,\CA).
\]
\end{lem}

\begin{proof}
The functors $H^*_{G,c}(A,-)$ and $H^*_{G,c}(X,(-)^X)$ are cohomological $\partial$-functors in the sense of \cite[Section 2.1]{Grothendieck} with naturally isomorphic functors in degree $0$
\[
H^0_{G,c}(A,-) \isom \Gamma_c^G(-)= \Gamma_c(-)^G\isom \Gamma_c^G((-)^X) \isom H^0_{G,c}(X,(-)^X).
\]
Since $H^*_{G,c}(A,-)$ consists of the right derived functors of the left exact functor $\Gamma_c^G(-)$, it is a universal $\partial$-functor. On the other hand, \fref{lem:effaceable} together with \cite[Proposition 2.2.1]{Grothendieck} imply that $H^*_{G,c}(X,(-)^X)$ is also a universal $\partial$-functor, hence it must be naturally isomorphic to $H^*_{G,c}(A,-)$.
\end{proof}

\begin{prop} \label{prop:long-exact-sequence-of-closed-subset}
	Let $G$ be a finite group, $\mathcal A$ a $G$-sheaf on a locally compact Hausdorff $G$-space $X$, $F\subseteq X$ a $G$-invariant closed subset of $X$, $U=X\setminus F$. Then we have a long exact sequence
	$$
	\cdots
	H_{G,c}^d\big(U;\mathcal A|_U\big)\to
	H_{G,c}^d\big(X;\mathcal A\big)\to
	H_{G,c}^d\big(F;\mathcal A|_F\big)\to
	H_{G,c}^{d+1}\big(U;\mathcal A|_U\big)\cdots
	$$	
\end{prop}
\begin{proof}
	The sheaves $\CA_U=(\CA|_U)^X$ and $\CA_F=(\CA|_F)^X$ can be included in a short exact sequence $0\to\mathcal A_U\to\mathcal A\to\mathcal A_F\to 0$ of $G$-sheaves on $X$, from which the cohomology functor $H^*_{G,c}$ produces a long exact sequence
	\[
	\cdots
	H_{G,c}^d\big(X;\mathcal A_U\big)\to
	H_{G,c}^d\big(X;\mathcal A\big)\to
	H_{G,c}^d\big(X;\mathcal A_F\big)\to
	H_{G,c}^{d+1}\big(X;\mathcal A_U\big)\cdots
	\]	
	in a natural way. The modules $H_{G,c}^i\big(X;\mathcal A_U\big)$ and $H_{G,c}^i\big(X;\mathcal A_F\big)$ in this sequence can be replaced by the isomorphic modules $H_{G,c}^i\big(U;\mathcal A|_U\big)$ and $H_{G,c}^i\big(F;\mathcal A|_F\big)$, respectively, as a consequence of \fref{lem:long-exact-sequence-isomorphism}.
\end{proof}	

\begin{proof}[Proof of \fref{prop:long-exact-sequence-of-closed-subset-cor}]
	First note that the statement makes sense as by \fref{prop:Smith-Floyd}
	the connected components of $F$ are {\cm}'s, so $F$ has
	an orientation sheaf, and all appearing orientation sheaves have a natural $G$-sheaf structure by \fref{prop:G-structure on orientaion_sheaf}.
	Applying \fref{prop:long-exact-sequence-of-closed-subset} to the $G$-sheaf $\CA=\Orientation[M]$ we obtain our claim	since $\Orientation[M]\big|_U=\Orientation[U]$,
	and, by \fref{prop:Smith-Floyd}, $\Orientation[M]\big|_F\isom\Orientation[F]$.
\end{proof}

\begin{prop}\label{prop:free_cross_trivial_action}
	Let $X$ be a Hausdorff space,
	$K$ a finite group acting continuously and freely on $X$,
        and denote by $f\colon X\to Y=X/K$ the quotient map.
        Let $\CA$ be a $K$-sheaf of $\BF$-vector spaces on $X$, where
        $\BF$ is a field.
	Let $L$ be another group
	and $\CM$ be an $\BF L$-module.
	Let $L$  act trivially both on $X$ and $\CA$. Then $G=K\times L$ acts naturally on $X$
	and on the sheaf $\CA\otimes_\BF\CM$, and for all $n\ge 0$, there are isomorphisms
	$$
	H^n_{G,c}\big(X;\CA\otimes_\BF\CM\big)
	\isom
	\bigoplus_{p+q=n}
	H^p_c\big(Y;f_*^{K}\CA\big) \otimes_\BF H^q\big(L;\CM\big),
	$$
	which are natural transformations with respect to
	both $\CA$ and $\CM$.
\end{prop}

\begin{proof}
        For simplicity we introduce the notation $\CF=f_*^K(\mathcal A)$.
	As $K$ acts freely and $L$ acts trivially on the Hausdorff space $X$, the functor $f_*^K$ is an equivalence between the category of $G$-sheaves of
	$\BF$-vector spaces on $X$ and the category of $L$-sheaves of $\BF$-vector spaces on $Y$. Hence we have a natural isomorphism
	\[
	H^*_{G,c}(X;\CB)=\rmR^*\Gamma_{c}^G(\CB)\isom \rmR^*\Gamma^K_{c}(f_{*}^K(\CB))=H^*_{L,c}(Y;f_{*}^K(\CB))
	\]
	for any $G$-sheaf $\CB$ of $\BF$-vector spaces on $X$. In particular,
	\begin{equation}
          \label{eq:2}
          H^*_{G,c}(X;\CA\otimes_\BF\CM)\isom H^*_{L,c}(Y;\CF\otimes_\BF\CM).
        \end{equation}
	Let $\CI$ be an injective sheaf of $\BF$-vector spaces on $Y$ (with trivial action of $L$ on $\CI$),
	and $\CJ$ an injective $\BF L$-module.
	Let $\iota\colon Y\to Y$ be the identity map.
	By \cite[Theorem~5.3.1]{Grothendieck},
	at each point $y\in Y$ we have
	\[
	\big(\rmR^j\iota^L_*(\CI\otimes_\BF\CJ)\big)_y = H^j(L;\CI_y\otimes_\BF \CJ) = \CI_y\otimes_\BF H^j(L;\CJ)=0
    \quad
	\text{for all }\ j>0.
	\]
	This implies that
	\[
	\rmR^j\iota^L_*(\CI\otimes_\BF\CJ) =
	\begin{cases}
	\CI \otimes_\BF \CJ^L &\text{for }\ j=0, \\
	0 &\text{for }\ j>0.
	\end{cases}
	\]
	Since tensor product with the vector space $\CJ^L$
	commutes with sheaf cohomology,
	\[
	H_{c}^i(Y; \CI\otimes_\BF\CJ^L ) \isom
	H_{c}^i(Y; \CI) \otimes_\BF\CJ^L =0
	\quad\quad
	\text{for all }\ i>0.
	\]
	Combining with the previous formula we obtain that
	\[
	H_{c}^i\big( Y;\rmR^j\iota^L_*(\CI\otimes_\BF\CJ )\big) =0
	\quad\quad
	\text{whenever }\ i+j>0.
	\]
	By \cite[Equation~(5.7.4)]{Grothendieck}, there is a spectral sequence
	\[
	I_t^{i,j}\Longrightarrow H_{L,c}^{i+j}(Y; \CI\otimes_\BF\CJ),\quad\quad I_2^{i,j}=	H_{c}^i\big( Y;\rmR^j\iota^L_*(\CI\otimes_\BF\CJ)\big),
	\]
which implies
	\[
	H_{L,c}^n( Y; \CI\otimes_\BF\CJ ) = 0
	\quad\quad
	\text{whenever }\ n>0.
	\]
	Finally let $\scrI^\bullet$ be an injective resolution of the sheaf $\CF$,
	and $\scrJ^\bullet$ be an injective resolution of the module $\CM$.
	By the above calculation, the complex
	$\scrI^\bullet\otimes_\BF\scrJ^\bullet$
	is a $\Gamma_{c}^L$-acyclic resolution of
	$\CF\otimes_\BF\CM$, hence
	$$
	H_{L,c}^n(Y, \CF\otimes_\BF\CM) \isom
	H^n\big(\Gamma_{c}^L(\scrI^\bullet\otimes_\BF\scrJ^\bullet)\big)\isom
	H^n\big( \Gamma_{c}(\scrI^\bullet)\otimes_\BF(\scrJ^\bullet)^L\big).
	$$
	The K\"unneth formula for complexes of vector spaces
	implies that
	\begin{eqnarray*}
		H_{L,c}^n(Y, \CF\otimes_\BF\CM)
		&\isom&
		\bigoplus_{i+j=n} H^i\big(\Gamma_{c}(\scrI^\bullet)\big) \otimes_\BF
		H^j\big((\scrJ^\bullet)^L\big)\\
		&\isom&
		\bigoplus_{i+j=n} H_{c}^i(Y;\CF)\otimes_\BF H^j(L;\CM).
	\end{eqnarray*}
	This isomorphism and \fref{eq:2} proves the proposition. Naturality is straightforward to check.
\end{proof}

\begin{proof}[Proof of \fref{prop:free-action-cohomology-cor}] If the previous proposition is applied to $X=M$, $\CA=\Orientation[M]$, and $\CM=\BF_p$, we obtain statement (b). All we need is the obvious isomorphism $f_*^K(\Orientation[M])\isom\Orientation[M/G]$. When $G=K$ and $L$ is a trivial group, case (b) reduces to case (a).
\end{proof}

\begin{proof}[Proof of \fref{prop:Leray-spectral-sequence}] The spectral sequence $\BE_t^{i,j}$ is a special case of Grothendieck's second spectral sequence $\mathrm{II}_t^{i,j}$ introduced in  \cite[Section~5.7]{Grothendieck}. We remark that this spectral sequence is a variant of Borel's spectral sequence (see \cite[Theorem~IV-9.2]{bredon2012sheaf}).
\end{proof}

\section{Bound on the rank of finite topological transformation groups}
\label{sec:Mann_Su}
In this section, we prove \fref{thm:Mann-Su_v0}.
The following lemma reduces the general case of finite transformation groups to that of elementary abelian ones.

\begin{lem}[Halasi, Podoski, Pyber, Szab\'o {\cite[Corollary~1.8]{p-group-stuff}}]
  \label{lem:rank-and-elementary-abelian-subgroups}
  If every elementary abelian subgroup of a finite group $G$ has rank at most $d$,
  then $G$ has rank at most $\frac12d^2+2d+1$.\qed
\end{lem}

The dimensions of the cohomology   vector spaces $H_c^*(M;\Orientation)$ of a manifold $M$ are controlled by the homology group $H_*(M;\BZ)$.

\begin{lem}\label{lem:universal_bound_on_mod_p_cohomology}
If $M$ is a manifold such that $\dim H_*(M;\BQ)= B$ is finite and
	$H_*(M;\BZ)$ has $\tau$ torsion elements, then $\dim H_c^*(M;\Orientation)\le B+2\tau$ for all primes $p$.
\end{lem}
\begin{proof}
	By the Poincar\'e duality \cite[Theorem V-9.2]{bredon2012sheaf},
	it is enough to prove that $\dim H_*(M;\BZ_p)\le B+2\tau$.

	By  the universal coefficient theorem,
	there is a split natural short  exact sequence
	\begin{equation}\label{eq:univ_coeff}
	0\to
	H_*(M;\BZ)\otimes \BZ_p \to
	H_*(M;\BZ_p)\to
	\Tor\big(H_*(M;\BZ),\BZ_p\big)\to
	0
	\end{equation}
	for all primes $p$, where
	\begin{equation}
	\label{eq:15}
	\Tor\big(H_*(M;\BZ),\BZ_p\big)=\{h\in H_*(M;\BZ):ph=0\},
	\end{equation}
	which has at most $\tau$ elements.
	Moreover, $H_*(M;\BZ)\otimes \BZ_p$ is an elementary abelian
	$p$-group of rank at most $\tau+B$.
	This implies that $\dim H_*(M;\BZ_p)\le B+2\tau$.
\end{proof}

We recall Borel's Fixed Point Formula \cite[Theorem XIII-4.3]{borel1960seminar}.

\begin{prop} [Borel]
	\label{prop:Borel-fixpoint-formula}
	Let $M$ be  a first countable $d$-{\cm}
	for some prime $p$, and
	let $G$ be an elementary abelian $p$-group
	acting continuously and effectively on $M$.
	Let $x\in \Fix{M}{G}$ be a fixed point of $G$.
	For a subgroup $H\le G$, denote by  $d(H)$ the dimension
	of the connected component of $x$ in $M^H$.
	Then
	$$
	d-d(G) = {\sum}_H \big( d(H) - d(G) \big),
	$$
	where $H$ runs through the subgroups of $G$ of index $p$.
\end{prop}

\begin{lem} \label{lem:Subgroups-of-Stabilisers}
	Let $G$ be an elementary abelian $p$-group,
	and $M$ be a connected effective {\Gdcm}.
	For a maximal stabilizer subgroup $H\in\Stab_{\max}(G,M)$,
	let $F$ be a connected component of $\Fix{M}{H}$.
	Then $\rank(H)\le d$, and
	$F$ has an $H$-invariant open neighbourhood
	$U\subseteq M$ such that
	\begin{enumerate}[\indent(a)]
		\item \label{item:15}
		there is a complete flag of subgroups
		$$
		\{1\}=H_0\lneq H_1\lneq\dots\lneq H_{\rank(H)}=H,
		\quad\quad
		\rank(H_i)=i,
		$$
		all $H_i$ belonging to $\Stab(H,U)\subseteq\Stab(G,M)$,
		\item \label{item:16}
		there are at most $d!$ such complete flags in $\Stab(H,U)$.
	\end{enumerate}
\end{lem}
\begin{proof}
	By \fref{prop:Smith-Floyd}, $F$ is a {\cm}.
	For a subgroup  $K\leq G$, let $Z(K)$ be the union of those
	connected components of $\Fix{M}{K}$ which are disjoint from
	$F$.
	By \fref{prop:Smith-Floyd},
	each connected component of $\Fix{M}{K}$ is a {\cm}, hence
	$\Fix{M}{K}$ is locally connected, and $Z(K)$ is closed in $M$. If $h\in
	H$, then $h(Z(K))=Z(hKh^{-1})$.
	
	Define the $H$-invariant open neighbourhood $U$ of $F$ by
	\[
	U=M\setminus \bigcup_{K\leq G}Z(K).
	\]
	If $x\in U$, then $x\notin Z(G_x)$ implies that $\Fix{M}{G_x}$ intersects $F$ at a point $y$. Then we have $(G_x\cup H)\subseteq G_y$, and $H\in \Stab_{\max}(G,M)$ yields $G_x\leq G_y=H$. Thus,  $\Stab(H,U)=\Stab(G,U)\subseteq\Stab(G,M)$.
	
	If $K\leq H$, then $\Fix{M}{K}\setminus Z(K)$ is the connected
	component of $\Fix{M}{K}$ containing $F$, and $\Fix{U}{K}$ is an open
	subset of $\Fix{M}{K}\setminus Z(K)$, therefore,
	$\Fix{U}{K}$ is a {\cm}
	and $\Fix{U}{K}\supseteq F$.
	
	Assume that we have $H_x\lneq H_y$ for the points $x,y\in U$. Then
	$\Fix{M}{H_y}\subseteq  \Fix{M}{H_x}$, and consequently
	$\Fix{M}{H_y}\setminus Z(H_y)$ is a proper closed  subset of
	$\Fix{M}{H_x}\setminus Z(H_x)$,
	which is a connected {\cm},
	and  $\Fix{M}{H_y}\setminus Z(H_y)$
	does not cover the point $x$. This gives that
	\[\dim_p(\Fix{U}{H_y})<\dim_p(\Fix{U}{H_x})\]
	by \cite[Corollary I-4.6]{borel1960seminar}.
	
	To prove \fref{item:15}, we construct $H_i$  by downward induction on
	$i$. We begin with $H_{\rank(H)}=H$ as required.
	Suppose that for some $i\ge 1$,
	the subgroup $H_i$ of rank $i$ is already constructed.
        Let $x\in M$ be a point whose stabilizer $E_x$ is $H_i$.
	Applying \fref{prop:Borel-fixpoint-formula}
	to $x$ we obtain that there are  nearby points $y\in M$
        whose stabilizer $H'=H_y$ is a subgroup of $H_i$ of rank $i-1$,
        and the number of such subgroups $H'$ is
        between $1$ and $d-\dim_p(\Fix{U}{H_i})$.
	Choose one of these subgroups for  $H_{i-1}$ to complete the induction step.
	
	The sequence
	\[
	\dim_p(\Fix{U}{H_0})>\dim_p(\Fix{U}{H_1})>\dots>\dim_p(\Fix{U}{H_{\rank(H)}})
	\]
	is strictly decreasing, hence $\rank(H)\le d-\dim_p(F)\le d$.
	
	Finally, we have seen above that for any fixed subgroup $H_i$,
	there are at most $d-\dim_p(\Fix{U}{H_i})\leq d-\dim_p(F)-\rank(H)+i$ possible choices for $H_{i-1}$. Thus, the total number of complete flags in $\Stab(H,U)$ is at most
	\[
	\prod_{i=1}^{\rank(H)}(d-\dim_p(F)-\rank(H)+i)\leq d!,
	\]
	as claimed.
\end{proof}

The following lemma is a variation of
\cite[Theorem 2.3]{Mann_Su}.
\begin{lem} \label{lem:Mann-Su-free}
	For all integers $d,B>0$, there is an integer $f(d,B)$ with the
	following property.
	Let $G$ be an elementary abelian $p$-group,
	and let $M$ be a {\Gdcm}.
	Assume that the $G$-action is free,
	and that $B=\dim H_c^*(M;\Orientation)<\infty$.
	Then the rank of $G$ is at most $f(d,B)$.
\end{lem}
\begin{proof} We show that $f(d,B)=B^2+d(B(d+1)-1)$ is a good choice for $f$.
	By assumption, $\big|\Aut(H_c^*(M;\Orientation)\big|<p^{B^2}$,
	hence $G$ has a subgroup $\tilde G$ of rank $\tilde r\ge\rank(G)-B^2$
	which acts trivially on $H_c^*(M;\Orientation)$.
	If $\tilde G=\{1\}$, then $\rank(G)\leq B^2$, and we are done.
	Assume that $\tilde G\neq\{1\}$.
	
	Consider the spectral sequence $\EE_t^{i,j}$ of
	\fref{prop:Leray-spectral-sequence}
	for the $\tilde G$-action on $M$.
	By assumption, $\EE_2^{i,j}=0$ for $j>d$, and
	\fref{prop:free-action-cohomology-cor} implies that
	$\EE_\infty^{i,j}=0$ for $i+j>d$.
\fref{lem:H_0_H^d_duality}  implies that $H_c^*(M;\Orientation)\neq0$.
	Let $s$ be the smallest integer such that $H_c^s(M;\Orientation)\neq0$.
	Then
	$$
	\dim \EE_2^{d+1,s}\le
	\sum_{k=1}^{d-s} \dim \EE_2^{d-k,s+k},
	$$
	so by the definition of $\EE_2$
	(see \fref{prop:Leray-spectral-sequence})
	and by \fref{prop:cohomology-of-elementary-abelian-p-groups}~\fref{item:12},
	we obtain
	$$
	\dim H^{d+1}(\tilde G;\BZ_p) \le
	\dim H^{d+1}\big(\tilde G;H_c^s(M;\Orientation)\big) \le
	$$
	$$
	\le
	\sum_{k=1}^{d-s} \dim H^{d-k}(\tilde G;\BZ_p)\cdot\dim H_c^{s+k}(M;\Orientation) \le
	$$
	$$
	\le
	(d-s)\,\dim H^{d}(\tilde G;\BZ_p)\cdot B \le
	dB\frac{d+1}{\tilde r+d}\dim H^{d+1}(\tilde G;\BZ_p).
	$$
	As $H^{d+1}(\tilde G;\BZ_p)\neq0$ by \fref{prop:cohomology-of-elementary-abelian-p-groups},
	we have $\tilde r\le dB(d+1)-d$.
	Then $\rank(G)\le \tilde r+B^2$ implies the lemma.
\end{proof}

\begin{thm}\label{thm:Mann-Su}
	For all integers $d,B>0$, there is an integer $r(d,B)$ with the
	following property.
	Let $G$ be an elementary abelian $p$-group,
	and $M$ be an effective  $G\curvearrowright$ $d$-\cm \! such that $\dim H_c^*(M;\Orientation)\le B$.
	Then the rank of $G$ is at most $r(d,B)$.
\end{thm}
\begin{proof}
	We show that the function $r(d,B)=\log (B!)+B(f(d,B)+ d)$ satisfies the requirement, where $f(d,B)=B^2+d(B(d+1)-1)$ is the function defined in \fref{lem:Mann-Su-free}.
	
	Let $M_1,\dots,M_C$ be the connected components of $M$. \fref{lem:H_0_H^d_duality} yields that $C\leq B$. $G$ permutes the components of $M$,
	let $\tilde G\le G$ be the subgroup of those elements
	which map each component $M_i$ into itself. Then $|G/\tilde G|\le C!\le B!$.
	
	Let $\tilde G_i$ be the image of $G$ in the homeomorphism group of $M_i$. The group $\tilde G_i$ acts effectively on $M_i$, and since the action of $G$ on $M$ is effective, the natural homomorphism $\tilde G\to \tilde G_1\times \dots\times \tilde G_C$ is injective.		
	
	Choose a maximal stabilizer $H_i\in \Stab_{\max}(\tilde G_i,M_i)$ for each $1\le i\le C$ and denote by $F_i$ the fixed point set $M_i^{H_i}$.
	
	The quotient group $\tilde G_i/H_i$ acts freely on $F_i$.
	\fref{prop:Smith-Floyd}
	implies that $\dim H_c^*(F_i;\Orientation)\le \dim H_c^*(M_i;\Orientation)\le B,$
	so \fref{lem:Mann-Su-free} gives $\rank \tilde G_i/H_i\le f(d,B)$.

	Moreover, $\rank H_i\le d$ by \fref{lem:Subgroups-of-Stabilisers}~\fref{item:15}, thus, we have
	\[\rank G\le \rank G/\tilde G+\sum_{i=1}^{C}(\rank \tilde G_i/H_i+\rank H_i)\le \log (B!)+B(f(d,B)+ d).\]
	This  proves the theorem.
\end{proof}

By \fref{lem:rank-and-elementary-abelian-subgroups} and \fref{lem:universal_bound_on_mod_p_cohomology}, \fref{thm:Mann-Su} completes the proof of \fref{thm:Mann-Su_v0}.

\section{Proof of \fref{thm:for-p-groups,topological-manifolds} for elementary $p$-groups}
\label{sec:proof-p-groups}

\begin{defn}
  For a topological space $X$, let $\Homeo(X)$ denote the group of all
  homeomorphisms of $X$ to itself.
\end{defn}

\begin{lem} \label{lem:one-dimensional-stabilizers}
  Let $M$ be a \dcm such that $B=\dim H_c^*(M;\Orientation)<\infty$. Let $r$
  be a natural number.
  Then there is a number $C_0(d,B,r)$ (independent of $p$) with
  the following property.
  If $E<\Homeo(M)$ is an elementary abelian
  $p$-subgroup  of rank $r$ such that for every $x\in M$
  the stabilizer $E_x$ is either trivial or isomorphic to $\BZ_p$, then we have
  $$
  \big|\Stab(E,M)\big|\le C_0(d,B,r).
  $$
\end{lem}
\begin{proof}
Set $$C_0(d,B,r):=B\sum_{i=0}^{d+1}\dim H^i(E;\BF_p)=B\sum_{i=0}^{d+1}\binom{ i+r-1}{r-1}=B\binom{r+d+1}{r}.$$
Let $\{L_1,\dots,L_k\}$ be the elements of $\Stab(E,M)$ isomorphic to $\BZ_p$.
Define $M_i=M^{L_i}$ and $M^*=\bigcup_iM_i$. Since $E$ acts freely
on $M\setminus M^*$, the inclusion $M^*\hookrightarrow M$ induces an
isomorphism
$$H_{E,c}^{d+1}(M^*;\Orientation)\isom H_{E,c}^{d+1}(M;\Orientation)$$
by \fref{prop:long-exact-sequence-of-closed-subset-cor} and \fref{prop:free-action-cohomology-cor}~(\ref{item:8}), so we have
$$\dim H_{E,c}^{d+1}(M^*;\Orientation)=\dim H_{E,c}^{d+1}(M;\Orientation).$$
The existence of a
spectral sequence converging to $H_{E,c}^*(M;\Orientation)$,
with second page $\BE_2^{i,j}=H^i(E;H^j_c(M;\Orientation))$,
(see \fref{prop:Leray-spectral-sequence}), and
\fref{prop:cohomology-of-elementary-abelian-p-groups} \fref{item:14} imply that
\begin{equation}
\label{eq:upper-bound}
\dim H_{E,c}^{d+1}(M;\Orientation)\leq B\sum_{i=0}^{d+1}\dim H^i(E;\BF_p)=C_0(d,B,r).
\end{equation}
The fixed point sets $M_1,\dots,M_r$ are disjoint, for if some point $x$ belonged
to $M_i\cap M_j$ with $i\neq j$, then $E_x$ would contain $L_iL_j\isom\BZ_p^2$, a contradiction.
For each $i$, the group $L_i$ acts trivially on $M_i$ and $E/L_i$ acts freely on $M_i$.
So the connected components of the quotient $N_i:=M_i/(E/L_i)$ are {\cm}'s and hence $N_i$ carries an orientation sheaf $\Orientation$.
Applying \fref{prop:free-action-cohomology-cor} \fref{item:9}, we have
\begin{align*}
H_{E,c}^{d+1}(M^*;\Orientation)  &=\bigoplus_{i=1}^kH_{E,c}^{d+1}(M_i^*;\Orientation)
\isom \bigoplus_{i=1}^k H_{L_i,c}^{d+1}(N_i;\Orientation) \\
&\isom \bigoplus_{i=1}^k \bigoplus_{u+v=d+1}H^u(L_i;\BF_p)\otimes H^v_c(N_i;\Orientation). 
\end{align*}
Let $N'_i$ be a connected component of $N_i$, and let $s=\dim N'_i$.
Then $H^s_c(N_i';\Orientation)\neq 0$. One can further
decompose the above expression for $H_{E,c}^{d+1}(M^*;\Orientation)$
in terms of contributions from each connected component of $N_i$,
and one of the summands is $H^{d+1-s}(L_i;\BF_p)\otimes H^s_c(N_i';\Orientation)$.
Since $L_i\isom\BZ_p$ and $s\leq d$ we have $H^{d+1-s}(L_i;\BF_p)\isom\BF_p$, so
$$H^{d+1-s}(L_i;\BF_p)\otimes H^s_c(N_i';\Orientation)\neq 0,$$
and hence $H_{E,c}^{d+1}(M_i^*;\Orientation)\isom H_{L_i,c}^{d+1}(N_i;\Orientation)\neq 0$. It follows that
$$\dim H_{E,c}^{d+1}(M^*;\Orientation)\geq k.$$
Combining this with \fref{eq:upper-bound} we obtain $k\leq C_0(d,B,r)$.
\end{proof}

\begin{lem} \label{lem:for-elementary-abelian-p-groups}
  Let $M$ be a \dcm such that $B=\dim H_c^*(M;\Orientation)<\infty$.
  Then there is a number $C_1(d,B)$ (independent of $p$)
  such that for every elementary abelian
  $p$-subgroup $E<\Homeo(M)$ we have
  $$
  \big|\Stab(E,M)\big|\le C_1(d,B).
  $$
\end{lem}
\begin{proof}
Let us first prove that for any natural number $r$, there is a number $C_1(d,B,r)$ (independent of $p$)
  such that for every elementary abelian
  $p$-subgroup $E<\Homeo(M)$ of rank $r$, we have
  $\big|\Stab(E,M)\big|\le C_1(d,B,r)$.

  We use induction on $r$.

  The case $r=1$ is obvious. Suppose that $r=2$, so let $E<\Homeo(M)$ be isomorphic to $\BZ_p^2$.
  Since $E$ acts effectively on $M$, it also acts effectively on
  $M\setminus M^E$, and hence we can identify $E$ with a subgroup of $\Homeo(M\setminus M^E)$.
  By \fref{prop:Smith-Floyd}, $\dim H^*_c(M\setminus M^E;\Orientation)\leq 2B$.
  Finally, the stabilizers of the action of $E$ on any point of $M\setminus M^E$ are either trivial
  or isomorphic to $\BZ_p$. Hence by \fref{lem:one-dimensional-stabilizers} we have
  $$\big|\Stab(E,M\setminus M^E)\big|\le C_0(d,2B,2).$$
  Since $\Stab(E,M)\setminus \Stab(E,M\setminus M^E)$
  contains at most one element (namely, $E$, in case $M^E\neq\emptyset$) we conclude in this case
  $$\big|\Stab(E,M)\big|\le C_0(d,2B,2)+1.$$

  We now prove the induction step. Assume that $r\geq 3$ and that the existence of the constant
  $C_1(d,B,r-1)$ has already been proved. Let $E<\Homeo(M)$ be an elementary abelian
  $p$-group of rank $r$. It will be convenient to look $E$ as an $r$-dimensional vector space
  over $\BF_p$. Let
  $$\Stab_{\ge 2}(E,M)=\{F\in\Stab(E,M)\mid \dim F\ge 2\}.$$
  We will first bound the size of $\Stab_{\ge 2}(E,M)$.
  Choose hyperplanes $H_1,\dots,H_r<E$ such that $H_1\cap\dots\cap H_r=\{0\}$.
  For every $j\geq 2$ let
  $$\CS_j=\{F\in \Stab_{\ge 2}(E,M)\mid F\cap H_1\nleq H_1\cap H_j\}.$$
  It is clear that any $F\in\CS_j$ satisfies
  $F=(F\cap H_1)+(F\cap H_j)$.
  Hence each $F\in\CS_j$ is uniquely characterized by its intersections
  with $H_1$ and $H_j$. Since we have
  $F\cap H_i\in\Stab(H_i,M)$ for $i=1,j$, the induction hypothesis implies that
  $$|\CS_j|\leq C_1(d,B,r-1)^2.$$
  But $\Stab_{\ge 2}(E,M)=\bigcup_{j\geq 2}\CS_j$; indeed, if
  $F\in \Stab_{\ge 2}(E,M)$ does not belong to $\bigcup_{j\ge 2}\CS_j$, then
  $F\cap H_1$ is contained in $H_1\cap H_2\cap\dots\cap H_r=\{0\}$, which
  implies $\dim F\cap H_1=0$, a contradiction. It follows that
  $$|\Stab_{\ge 2}(E,M)|\leq (r-1)C_1(d,B,r-1)^2.$$
  Now let
  $$M^*_1=M\setminus\bigcup_{F\in \Stab_{\ge 2}(E,M)}M^F.$$
  Applying repeatedly inequality \fref{eq:Smith_Floyd_2}, we obtain
  $$\dim H^*_c(M^*_1;\Orientation)\leq 2^{(r-1)C_1(d,B,r-1)^2}B.$$
  We have
  $$\Stab(E,M)=\Stab(E,M^*_1)\cup \Stab_{\ge 2}(E,M),$$
  and the stabilizers of the action of $E$ on $M^*_1$ are either trivial or isomorphic
  to $\BZ_p$. Hence by \fref{lem:one-dimensional-stabilizers}
  the size of $\Stab(E,M)$ is at most
  $$C_1(d,B,r):=C_0(d,2^{(r-1)C_1(d,B,r-1)^2}B,r)+(r-1)C_1(d,B,r-1)^2.$$
  This finishes the induction step, so the proof of the existence of the constants
  $C_1(d,B,r)$ is now complete.

  By \fref{thm:Mann-Su_v0}
  the rank of the elementary $p$-subgroups $E<\Homeo(M)$ is bounded independently
  of $p$, say by a constant $R$. Hence the number $C_1(d,B)=\max_{1\leq r\leq R}C_1(d,B,r)$
  has the property stated in the lemma.
\end{proof}

\section{Proof of \fref{thm:for-p-groups,topological-manifolds} for general $p$-groups}
\label{sec:proof-general}

Our proof is inductive.
In order to make the induction working,
we prove a slightly more general statement:

\begin{thm} \label{thm:for-p-groups,cohomology-manifolds}
  Let $M$ be a \dcm such that $B=\dim H_c^*(M;\Orientation)<\infty$.
  Then there is a number $\tilde C=\tilde C(d+B)$ (independent of $p$)
  with the following property.
  Each finite $p$-group $G\le\Homeo(M)$
  has a characteristic subgroup $H\le G$ of index at most $\tilde C$
  containing the center $\CZ(G)$
  such that
  $$
  \big|\Stab(H,M)\big|\le \tilde C.
  $$
\end{thm}

\begin{proof}
  The proof is by induction on $d$ and $B$.
  If $d=0$, then $|M|\le B$, hence $|G|\le  B!$ and
  $\big|\Stab(G,M)\big|\le 2^{B!}$, so the statement holds.

  Now assume that $d>0$, and the statement holds for \cm's of
  dimension $<d$, and for each \dcm \ N with $\dim H_c^*(N;\Orientation)<B$.

  Let $M$ be as in the statement.
  To complete the induction step,
  we need three claims:

  \begin{claim} \label{claim:intersecting-with-an-abelian-subgroup}
    There is a number $C_2(d+B)$ (independent of $p$)
    such that every finite $p$-group $G<\Homeo(M)$
    has a characteristic subgroup $H\le G$ of index at most $C_2(d+B)$,
    containing $\CZ(G)$, such that
    $$
    \Big|\left\{K\in\Stab(H,M)\,\Big|\, K\cap \CZ(G)\ne\{1\}\,\right\}\Big|
    \le C_2(d+B).
    $$
  \end{claim}
  \begin{proof}
    Let $E\le\CZ(G)$ be the subgroup of elements of order $p$.
    First we pick a subgroup $L\in\Stab(E,M)$ different from $\{1\}$.

    Denote by $M^L_s$ the union of the $s$-dimensional connected components of $M^L$.
    By assumption $M^L\ne M$, hence if $M^L_s\neq \emptyset$, then either $s<d$  and $\dim H_c^*(M^L_s;\Orientation)\leq \dim H_c^*(M^L,\Orientation)\le B$  by \fref{prop:Smith-Floyd} or $s=d$ and $\dim H_c^*(M^L_s;\Orientation)<\dim H_c^*(M,\Orientation) = B$.

    In both cases the induction hypothesis applies to $M^L_s$ and the image of $G$ in $\Homeo(M^L_s)$, and gives us a subgroup $H_s\le G$
    of index at most $\tilde C(d+B-1)$ containing $\CZ(G)$ such that
    $$
    \Big|\Stab(H_s,M^L_s)\Big| \le \tilde C(d+B-1).
    $$
    Let $H$ be the intersection of all subgroups of $G$
    of index at most $\tilde C(d+B-1)$ containing $\CZ(G)$.
    It is characteristic, and setting
    \begin{align*}
      \CS_{H,L} &= \left\{K\in\Stab(H,M)\,\big|\, K\cap E = L\right\}\\&\subseteq \bigcup_{s=0}^d\left\{K\cap H\,\big|\, K\in \Stab(H_s,M^L_s)\right\},
    \end{align*}
    we have
    $$
    \big|\CS_{H,L}\big|\le (d+1)\tilde C(d+B-1).
    $$
    Moreover,
    by \fref{thm:Mann-Su_v0} the rank of $G$ is at most $r(d,B)$,
    so by \cite[Corollary 1.1.2]{lubotzky-Segal-SubgroupGrowth}
    the number of subgroups of $G$ of index at most $\tilde C(d+B-1)$
    is bounded by
    \[
      x(d+B) = \tilde C(d+B-1)^2\left(\tilde C(d+B-1)!\right)^{r(d,B)-1} ,
    \]
    and the index of $H$ in $G$ is at most $\tilde C(d+B-1)^{x(d+B)}$.
    Since
    $$
    \left\{K\in\Stab(H,M)\,\Big|\, K\cap \CZ(G)\ne\{1\}\,\right\} =
    \bigcup_{\genfrac{}{}{0pt}{}{L\in\Stab(E,M),}{ L\ne\{1\}}}\CS_{H,L},
    $$
    \fref{lem:for-elementary-abelian-p-groups}
    implies the claim.
  \end{proof}

  \begin{claim} \label{claim:for-large primes}
    If $p>C_2(d+B)$,
    then for each finite $p$-subgroup $G<\Homeo(M)$,
    every stabilizer $K\in\Stab(G,M)$ with $|K|>1$
    intersects $\CZ(G)$ non-trivially:
    $$
    K\cap\CZ(G)\ne\{1\}.
    $$
  \end{claim}
  \begin{proof}
    We prove this by induction on $|G|$.
    To begin with, the claim holds if $G$ is abelian.
    For the induction step
    we assume that the claim holds for all groups smaller than $G$.

    Suppose there is a subgroup $K\in\Stab(G,M)$
    with $K\cap\CZ(G)=\{1\}$.

    Let $\CZ_0=\{1\}\le\CZ_1\le\CZ_2\le\dots$ denote the upper central
    series of $G$, i.e.
    $\CZ_{k+1}/\CZ_k=\CZ\big(G/\CZ_k\big)$ for all $k\ge0$.
    Since $G$ is a $p$-group, the union of these subgroups is $G$.
    In particular, one of them contains $K$.

    Let $t$ be the largest index such that $K\cap\CZ_t=\{1\}$.
    By assumption $t\ge1$.
    Choose an element $h\in K\cap\CZ_{t+1}$ of order $p$.
    Let $L$ be the subgroup generated by $\CZ_t$ and $h$.
    By construction $\gen{h}=K\cap L\in\Stab(L,M)$.

    Since $L/\CZ_t$ is central in $G/\CZ_t$,
    its pre-image $L$ must be normal in $G$.
    Since $\CZ_t/\CZ_{t-1}=\CZ\big(G/\CZ_{t-1}\big)$
    and $L/\CZ_t$ is cyclic,
    we obtain that $L/\CZ_{t-1}$ is abelian,
    and $G/\CZ_{t-1}$ is non-abelian.
    In particular, $L\ne G$.

    By the induction assumption the claim is valid for $L$,
    and combining this with
    \fref{claim:intersecting-with-an-abelian-subgroup}
    we obtain a subgroup $\tilde L\le L$ of index less than $p$
    such that $\big|\Stab(\tilde L,M)\big|\le C_2(d+B)+1$
    (the $+1$ stands for the subgroup $\{1\}$).
    But $L$ is a $p$-group, hence $\tilde L=L$, and
    $$
    \big|\Stab(L,M)\big|\le C_2(d+B)+1
    $$
    Since $L$ is normal, all $G$-conjugates of the subgroup $\gen h$ belong to
    $\Stab(L,M)$,
    hence there are at most $C_2(d+B)$ such conjugates.

    On the other hand,
    the number of $G$-conjugates of any subgroup of $L$ is
    either $1$, or divisible by $p$.
    Since $p$ is large, we obtain that $\gen h$ is $G$-invariant, i.e.,
    it is normal in $G$.

    Since $\gen h$ is a cyclic group of order $p$, the size of its
    automorphism group is not divisible by $p$.
    This implies the $G$ acts trivially on $\gen h$.
    In other words, $\gen h\le \CZ(G)$, a contradiction.
    This completes the induction step.
  \end{proof}

  The following is a result of Alperin \cite{alperin1964centralizers}
  (see also \cite[Satz~12.1]{huppert2013endliche}:
  \begin{prop} \label{prop:Alperin-centralizers}
    If $E$ is a subgroup of a $p$-group $G$,
    maximal subject to being normal abelian and of exponent at most $p^n$,
    then any element of order at most $p^n$ which centralizes $E$ lies
    in $E$, unless perhaps $p=2$ and $n=1$.
  \end{prop}

  \begin{claim} \label{claim:for-small-primes}
    If $p\le C_2(d+B)$,
    then each finite $p$-subgroup $G<\Homeo(M)$
    has a characteristic subgroup $H\le G$ of index at most $C_3=C_3(d+B)$, containing $\CZ(G)$, such that
    $$
    \big|\Stab(H,M)\big|\le C_2(d+B).
    $$
  \end{claim}
  \begin{proof}
    By \fref{thm:Mann-Su}, the rank of $G$ is bounded by $r(d,B)$.

    If $p=2$, then we set $n=2$, otherwise we set $n=1$.
    Let $E\le G$ be a subgroup
    maximal subject to being normal abelian and of exponent $p^n$.
    By assumption, $|E|\le C_2(d+B)^{2r(d,B)}$,
    hence the index of the centralizer $\CC_G(E)$ is also bounded
    by some function $C_3(d+B)$.

    The center of $\CC_G(E)$ contains $\CZ(G)$ and $E$.
    By \fref{prop:Alperin-centralizers} we know that
    in $\CC_G(E)$ each element of order $p$ is contained in $E$.
    In particular,
    if $K\le\CC_G(E)$ and $K\ne\{1\}$,
    then $K\cap E\ne\{1\}$.

    Applying
    \fref{claim:intersecting-with-an-abelian-subgroup}
    to the group $\CC_G(E)$
    we obtain a subgroup $\tilde H\le\CC_G(E)$ of bounded index,
    containing $\CZ(G)$ and $E$,
    such that $\big|\Stab(\tilde H,M)\big|$ is bounded.
    Then $|G:\tilde H|=|G:\CC_G(E)|\cdot |\CC_G(E):\tilde H|$ is bounded.

    Finally, let $H\le\tilde H$ be the intersection of all subgroups
    of $G$ containing $\CZ(G)$ whose index is $|G:\tilde H|$.
    It is characteristic,
    by \cite[Corollary 1.1.2]{lubotzky-Segal-SubgroupGrowth}
    it has bounded index, it contains $\CZ(G)$, and
    $\big|\Stab(H,M)\big|\le\big|\Stab(\tilde H,M)\big|$.
    This proves the claim.
  \end{proof}

  \subsection{Completing the proof of
    \fref{thm:for-p-groups,cohomology-manifolds} }
  \label{sec:completing-proof}

  \fref{claim:intersecting-with-an-abelian-subgroup}
  gives us a subgroup $H\le G$ of bounded index, containing $\CZ(G)$,
  such that
  $$
  \Big|\left\{K\in\Stab(H,M)\,\Big|\, K\cap \CZ(G)\ne\{1\}\,\right\}\Big|
  \le C_2(d+B).
  $$
 Since $\CZ(G)\le\CZ(H)$, 
depending on the value of $p$,
  either
  \fref{claim:for-large primes}
  or
  \fref{claim:for-small-primes}
  completes the induction step.
\end{proof}

\begin{proof}[Proof of \fref{thm:for-p-groups,topological-manifolds}]
 As $\dim H_c^*(M;\Orientation)$ is bounded in terms of $H_*(M;\BZ)$ by \fref{lem:universal_bound_on_mod_p_cohomology},
 the statement follows from \fref{thm:for-p-groups,cohomology-manifolds}.
\end{proof}

\section{Proof of \fref{cor:free-part-has-bounded-homology}}

We need some topological information about the fixed point structure
of finite $p$-group actions.

\begin{lem}\label{lem:topological-invariants-of-fixpoint-sets-and-complements}
  Let $M$ be a topological space whose connected components are \cm-s
  such that $B=\dim H_c^*(M;\Orientation)<\infty$.
  Let $H$ be a finite $p$-subgroup acting continuously on $M$.
  For each subset $\CS\subseteq\Stab(H,M)$ and each subgroup $K\le H$,
  we define the following subsets of $M$.
  \[
    F_\CS = \bigcup_{L\in\CS} M^L
    \quad\quad\text{and}\quad\quad
    U^K_\CS= M^K\setminus F_\CS.
  \]
  Then
  \[
    \dim H_c^*(U^K_\CS;\Orientation) \le 2^{|\CS|}B.
  \]

\end{lem}
\begin{proof}
  We argue by induction on the size of $\CS$.
  For $\CS=\emptyset$ the statement holds by \fref{prop:Smith-Floyd}~\fref{eq:Smith_Floyd_1}.
  Assume now that it holds for some $\CS$, and consider the subset
  $\CS'=\CS\cup\{L\}$ for some $L\in\Stab(H,M)\setminus\CS$.
  Then
  \[
    U^K_{\CS'} = U^K_\CS\setminus M^L = U^K_\CS\setminus U^{\gen{K,L}}_\CS.
  \]
  The long exact sequence
  for cohomology with compact support
  corresponding to the closed subset
  $U^{\gen{K,L}}_\CS\subseteq U^K_\CS$
  (see \cite[II-10.3]{bredon2012sheaf}),
  \fref{prop:Smith-Floyd}\fref{eq:O_F},
  and the induction hypothesis imply that
  \begin{align*}
    \dim H_c^*(U^K_{\CS'};\Orientation) &\le
    \dim H_c^*(U^K_{\CS};\Orientation) + \dim H_c^*(U^\gen{K,L}_{\CS};\Orientation) \\ &\le 2^{|\CS|}B + 2^{|\CS|}B = 2^{|\CS'|}B.
  \end{align*}
  This completes the induction step.
\end{proof}

\begin{cor} \label{cor:for-p-groups,cohomology-manifolds}
  Let $M$ be a \dcm such that $B=\dim H_c^*(M;\Orientation)<\infty$
  and $G$ a finite $p$-subgroup acting continuously on $M$.
  Then $G$ has a characteristic subgroup $H\le G$ of index at most
  $\breve C = \breve C(d,B)$ containing the center of $G$
  such that for each $K\in\Stab(H,M)$, the connected components of the
  locally closed subset
  $\mathring M^K=\{x\in M\;|\; H_x=K\}$ are \cm-s and
  \[
    \dim H^c_*(\mathring M^K;\BF_p) =
    \dim H_c^*(\mathring M^K;\Orientation) \le
    \breve C.
  \]
\end{cor}
\begin{proof}
  For any subgroup $K\le G$, the set
  $\mathring M^K$ is an open subset of $M^K$,
  hence its connected components are \cm-s by \fref{prop:Smith-Floyd}.
  Then
  $\dim H^c_*(\mathring M^K;\BF_p) =
  \dim H_c^*(\mathring M^K;\Orientation)$
  by the Poincar\'e duality (which is valid with Borel-Moore homology
  with compact support, see \cite[V-9.2]{bredon2012sheaf}).

  With the notation of \fref{thm:for-p-groups,cohomology-manifolds},
  let $\hat C = \max_{d'\le d}\tilde C(d'+B)$,
  and let $H$ be the intersection of all subgroups of $G$ of index at
  most $\hat C$ containing the center of $G$.
  This is a characteristic subgroup of $G$.
  By \fref{thm:Mann-Su}, the rank of $G$ is bounded,
  so the index of $H$ in $G$ is also bounded in terms of $d$ and $B$.

  We apply
  \fref{lem:topological-invariants-of-fixpoint-sets-and-complements}
  to a subgroup $K\in\Stab(H,M)$ and the subset
  $$
  \CS =
  \big\{L\in\Stab(H,M)\;\big|\;L>K\big\}.
  $$
  In this case $U^K_\CS=\mathring M^K$ and
  $|\CS|<\big|\Stab(H,M)\big|\le\hat C$,
  so we obtain a bound on
  $\dim H_c^*(\mathring M^K;\Orientation)$ depending on $d$ and $B$.
\end{proof}

\begin{proof}[Proof of {\fref{cor:free-part-has-bounded-homology}} and
  {\fref{rem:free-part-has-bounded-homology}}]
   By \fref{lem:universal_bound_on_mod_p_cohomology}, $\dim H_c^*(M;\Orientation)$ is bounded in terms of $H_*(M;\BZ)$. Then \fref{cor:for-p-groups,cohomology-manifolds} gives us a subgroup $H$ such that
  $\dim H^c_*(\mathring M^K;\BF_p) =
  \dim H_c^*(\mathring M^K;\Orientation)$
  is bounded  in terms of $\dim(M)$
  and $H_*(M;\BZ)$  for all $K\in\Stab(H,M)$.
  If $K$ is minimal in $\Stab(H,M)$
  then $\mathring M^K$ is an open submanifold in $M$,
  hence its Borel-Moore homology with compact support is isomorphic to
  its singular homology.
\end{proof}

\bibliographystyle{amsplain}
\bibliography{Stabilizers}

\end{document}